\documentclass{amsart}
\usepackage{amssymb,amsmath}
\usepackage{amsthm}
\usepackage{hyperref}

\makeatletter
\def\LaTeX{\leavevmode L\raise.42ex
\hbox{\kern-.3em\size{\sf@size}{0pt}\selectfont A}\kern-.15em\TeX}
\def\@currentlabel{2.1}\label{e:dispaa}
\def\@currentlabel{2.21}\label{e:dispau}
\def\@currentlabel{2.22}\label{e:dispav}
\def\@currentlabel{2.23}\label{e:dispaw}
\def\@currentlabel{2.24}\label{e:dispax}
\def\theequation{\thesection.\@arabic\c@equation}
\makeatother

\newtheorem{thm}{Theorem}[section]
\newtheorem{lem}{Lemma}[section]
\newtheorem{prop}{Proposition}[section]
\newtheorem{defn}{Definition}[section]

\numberwithin{equation}{section}

\newcommand{\al}{\alpha}                \newcommand{\lda}{\lambda}
                \newcommand{\pa}{\partial}
\newcommand{\va}{\varepsilon}           \newcommand{\ud}{\mathrm{d}}
\newcommand{\be}{\begin{equation}}      \newcommand{\ee}{\end{equation}}
                \newcommand{\w}{\omega}

\def\dlim{\displaystyle\lim}
\def\dsup{\displaystyle\sup}

\def\dinf{\displaystyle\inf}

\begin{document}

\title[Harnack Inequality]{A Harnack inequality for fractional Laplace equations with lower order terms}

\author{Jinggang Tan}
\address{J.T., Departamento de  Matem\'atica,
Universidad T\'{e}cnica Federico Santa Mar\'{i}a,  Avda. Espa\~na 1680,
Valpara\'{\i}so, Chile}

\email{jinggang.tan@usm.cl}

\author{Jingang Xiong}
\address{J.X., School of Mathematical Sciences, Beijing Normal University, Beijing 100875, China}
\email{jxiong@mail.bnu.edu.cn}


\thanks{ J.T. was supported
by Fondecyt \# 11085063.
J.X. was supported by CSC program for visiting Department of Mathematics, Rutgers University.}


\date{\today}

\begin{abstract}
We establish a Harnack inequality of fractional Laplace equations without imposing sign condition on the
coefficient of zero order term via the Moser's iteration and John-Nirenberg inequality.
\end{abstract}

\maketitle

\bigskip

\section{Introduction}
\bigskip

\noindent This note is devoted to a Harnack inequality of Laplace equations without imposing sign condition on the
coefficient of zero order terms.

The fractional Laplacians $(-\Delta)^{\sigma}$, $0<\sigma<1$, which are the infinitesimal
generators in stable L\'{e}vy stable processes, are given by
 the Fourier transform ${\mathcal{F}}$ as follows: for $u\in H^{\sigma}(\mathbb{R}^{n})$, $n\geq 2$,
 \begin{equation}\label{eqn-flap}
 {\mathcal{F}}((-\Delta)^{\sigma}u)(\xi):=|\xi|^{2\sigma}{\mathcal{F}}({u}(\xi))\;\xi\in \mathbb{R}^{n}.
 \end{equation}

Caffarelli and Silvestre  \cite{CS2} introduced fractional extension
 $v\in D^{1,2}_{\sigma}(\mathbb{R}_{+}^{n+1})$ of $v(x,0)=u(x)$   satisfying
 \begin{equation}\label{eqn-csi}
\int_{0}^{\infty}\int_{\mathbb{R}^{n}}|\nabla v(x,t)|^{2}\,t^{1-2\sigma}dtdx
=c_{\sigma}\int_{\mathbb{R}^{n}}|\xi|^{2\sigma}|{\mathcal{F}}(u)(\xi)|^{2}\,d\xi,
 \end{equation}
where $c_{\sigma}^{-1}=2^{-1}(4\pi)^{2\sigma}\Gamma(2-2\sigma)$. Then the
factional Laplacians are realized by the Dirichlet-Neumann map of $v$
\begin{equation}\label{eqn-DNm}
(-\Delta)^{\sigma}u(x)=-c_{\sigma}\lim\limits_{t\rightarrow
0} t^{1-2\alpha}v_{t}.
\end{equation}

Let $B_r\subset \mathbb{R}^n$ be the ball centered at origin with radius $r$. Our main result is
\begin{thm}\label{thm1}
Let $u \in H^{\sigma}(\mathbb{R}^{n})$ be nonnegative in $\mathbb{R}^n$ and $C^2(B_1)\cap C^{1}(\overline{B_{1}})$.
Suppose that $u(x)$ satisfies
\be\label{1.1}
(-\Delta)^\sigma u(x)=a(x)u(x)+b(x) \quad \mbox{in } B_1,
\ee
where $a(x), b(x)\in L^{\infty}(B_1)$. Then
\[
\sup_{B_{1/2}} u\leq C(\inf_{B_{1/2}} u+\|b\|_{L^\infty(B_1)}),
\]
where $C>0$ depends only on $n,\sigma,\|a\|_{L^\infty(B_1)}$.
\end{thm}

To prove it, we establish a Harnack inequality for the equivalent problem as follow.

Let $X=(x,t)\in \mathbb{R}^{n+1}$, $Q_R=B_R\times (0,R)\subset \mathbb{R}^{n+1}$ and $\pa' Q_R=B_{R}\times \{0\}$.
Define
\[
H(t^{1-2\sigma},Q_R):=\left\{U\in H^1(Q_R): \int_{Q_R}t^{1-2\sigma}(U^2+|\nabla U|^2)\,\ud X<\infty \right\}.
\]
\begin{thm}\label{thm3.1}
 Let $U\in H(t^{1-2\sigma},Q_1)$ be nonnegative solution $C^{2}(Q_{1})\cap C^{1}(\overline{Q_1})$ of
\begin{equation}\label{eqn-ep}
\begin{cases}
\mathrm{div}(t^{1-2\sigma}\nabla U(X))= 0\quad &\mbox{in } Q_1 \\
-\dlim_{t\rightarrow 0^+}t^{1-2\sigma}\pa_tU(x,t)= a(x)U(x,0)+b(x)\quad  &\mbox{on } \pa' Q_1.
\end{cases}
\end{equation}
Suppose $a, b\in L^{\infty}(B_1)$. Then
\[
 \dsup_{\overline Q_{1/2}}U\leq C(\dinf_{\overline Q_{1/2}}U+\|b\|_{L^\infty(B_1)}),
\]
where $C>0$ depends only on $n,\sigma$ and $\|a\|_{L^\infty(B_1)}$.
\end{thm}

The main feature is that we do not assume the sign condition of $a(x)$.
Previously, in the case $a(x)\equiv 0$,
Bass and Levin \cite{Bas02} establish the Harnack inequality for nonnegative
functions of a class of  symmetric stable processes that are harmonic
 with respect to these processes, see also \cite{ChenSong98} by Chen and Song.
 The analytic method was given by Caffarelli and Silvestre  \cite{CS2},
 by employing the fractional extension of fractional harmonic functions.

 We here establish the Harnack inequality as in Theorem \ref{thm3.1} by the Moser iteration.
 The proof bases on the properties of the weighted Sobolev space developed by Fabes, Kenig and Serapioni \cite{FKS}
 and the John-Nirenberg inequality in $A_{2}$ weighted $BMO$ space obtained by  Muchenhoupt and Wheeden \cite{MW}.

 If $\sigma=\frac{1}{2}$, the result is due to Han and Li \cite{HL}.
 After we complete our manuscript, we observe that
 Theorem \ref{thm3.1},  the Harnack inequality for $b\equiv 0$,
 has been shown recently  by Cabre  and Sire \cite{CS10} through
 making even extension and using the result
 of Fabes, Kenig and Serapioni~\cite{FKS}.
 But our proof has independent interest.

On the other hand, since the fractional Laplacian is a nonlocal operator,
the condition $u\geq 0$ in $\mathbb{R}^n$ cannot be relaxed to
$u\geq 0$ in $B_1$. In fact, we need all information in the complement
of $B_{1}$.  For example, see an counterexample of the case $a\equiv b \equiv 0$ in \cite{K}  by Kassmann.
By the Dirichlet-Neumann map, we transform (\ref{1.1}) to the local problem in $\mathbb{R}^{n+1}_{+}$,
which grantees the identity (\ref{eqn-csi}).
 The nonnegative assumption of $u$ implies that its fractional extension $v$ is nonnegative
 in the half space $\mathbb{R}^{n+1}_{+}$. Thus, $v\ge 0$ in all of cubes $Q_{R}, R>0$.
 Therefore, we can obtain the desired Harnack inequality by studying the local version (\ref{eqn-ep}).

The paper is organized as follows. In Section \ref{s2},
 we demonstrate some properties in the weighted Sobolev spaces.
The proofs of Theorem \ref{thm1}, \ref{thm3.1} are given in Section~\ref{s3}.

\bigskip

\section{Preliminaries}
\label{s2}
\bigskip

In this section, we shall present some important weighted inequalities.

Denote $Q_R=B_R\times (0,R)\subset \mathbb{R}^{n+1}$, $\pa' Q_R=B_{R}\times \{0\}$ and $\pa'' Q_R=\pa Q_R\setminus \pa'Q_R$.
We use capital letters like $X=(x,t), Y=(y,s)$ to represent points in $\mathbb{R}^{n+1}$.

Let us recall the definition of $A_2$ class.

\begin{defn}
 Let $\w(X)$ be a nonnegative measurable function in $\mathbb{R}^{n+1}$. We say $\w$ being of the class $A_2$ if
there exists a constant $C_{\w}$ such that for any ball $B\subset \mathbb{R}^{n+1}$
\[
 \left(\frac{1}{|B|}\int_{B}\w(X)\,\ud X\right)\left(\frac{1}{|B|}\int_{B}\w^{-1}(X)\,\ud X\right)\leq C_\w,
\]
where $|\cdot|$ is the Lebesgue measure.
\end{defn}

\begin{lem}\label{prop2.1}
Let $f(X)\in C^1_c(Q_R\cup \pa' Q_R)$ and $\w(X)\in A_2$.
Then there exist constants $C$ and $\delta>0$ depending only $n$ and $C_\w$ such that for any $1\leq k\leq \frac{n+1}{n}+\delta$
\be \label{2.1}
\left(\frac{1}{\w(Q_R)}\int_{Q_R}|f|^{2k}\w \,\ud X\right)^{1/2k}\leq C R \left(\frac{1}{\w(Q_R)}\int_{Q_R}|\nabla f|^{2}\w \,\ud X\right)^{1/2},\ee
where $\w(Q_R)=\int_{Q_R}\w(X)\,\ud X$.
\end{lem}

\begin{proof}
 The proof of this Lemma is similar to that of Theorem 1.2 in \cite{FKS}. The following inequality is the only thing we need to show.
\be \label{2.2}
|f(X)|\leq \frac{2}{\w_{n}}\int_{Q_R}\frac{|\nabla f(Y)|}{|X-Y|^{n}}\,\ud Y,\quad \mbox{for any } X\in Q_R,
\ee where $\w_n$ is the area of the sphere $\mathbb{S}^n$.

Extend $f$ to be zero outside $Q_R$. Let $X\in Q_R$, then (\ref{2.2}) follows from
\be \label{2.3}
f(X)=\frac{2}{\w_{n}}\int_{\mathbb{R}_-^{n+1}}\frac{\nabla f(X-Y)\cdot Y}{|Y|^{n+1}}\,\ud Y.
\ee
Since $X-Y\in \mathbb{R}^{n+1}_+$, $\nabla f(X-Y)$ makes sense.
Let $\xi\in \mathbb{S}^{n}_-$, the south half sphere. For $t>0$, note that
\[f(X)=\int^\infty_0-\frac{\pa}{\pa t}f(X-\xi t)\,\ud t=\int^\infty_0\nabla f(X-\xi t)\cdot \xi\,\ud t.
\]
We integrate the above over $\xi$ ranging on the south half sphere. This gives
\[
f(X)= \frac{2}{\w_{n}}\int_{\xi\in \mathbb{S}^{n}
_-}\int_0^\infty \nabla f(X-\xi t)\cdot \xi\,\ud t\ud \xi.
\] Identity (\ref{2.3}) follows from coordinate changing.
\end{proof}

Next we quote the following weighted Poincar\'e inequality which can be found in \cite{FKS}.

\begin{lem} \label{prop2.2} Let $f\in C^{1}(Q_R)$, then any $1\leq k\leq \frac{n}{n-1}+\delta$, we have
\[
 \left(\frac{1}{\w(Q_R)}\int_{Q_R}|f-f_{R,\w}|^{2k}\w\,\ud X\right)^{1/2k} \leq CR
 \left(\frac{1}{\w(Q_R)}\int_{Q_R}|\nabla f|^{2}\w\,\ud X\right)^{1/2},
\]
where $f_{R,\w}=\frac{1}{\w(Q_R)}\int_{Q_R}f\w$.
\end{lem}

Finally, we prove the following trace embedding result.

\begin{lem}\label{prop2.3}
 Let $f(X)\in C^1_c(Q_R\cup \pa' Q_R)$ and $\al\in (-1,1)$. Then
there exists a positive constant $\delta$ depending only on $\al$ such that
\be\label{2.4} \int_{\pa'Q_R}|f|^2\leq \va\int_{Q_R}|\nabla f|^2t^{\al}+\frac{C(R)}{\va^{\delta}} \int_{Q_R}| f|^2t^{\al},\ee
for any $\va>0$.
\end{lem}

\begin{proof}
For any  $1<p<\infty$, we have
\be \label{2.5}
\begin{split}
 \int_{\pa' Q_R}|f|^p&= -\int_{Q_R}\pa_t|f|^p=-\int_{Q_R} p |f|^{p-1}\mbox{sgn}f \pa_tf
\\&\leq \va \int_{Q_R}|\nabla f|^p+C\va^{-\frac{1}{p-1}}\int_{Q_R}|f|^p.
\end{split}
\ee

Next, we claim for $0<\al<1$ and any $\lda>-1$
\be\label{2.6}
\int_{Q_R}|f|^2t^{\lda}\leq C(\lda,\al)\int_{Q_R}|\nabla f|^2t^{\al}.
\ee
In fact, by the H\"older inequality
\[
 \begin{split}
f^2(x,t)&= (\int^R_t\pa_tf(x,s)\,\ud s)^2\leq \int^R_ts^{-\al}\,\ud s\int^R_t|\pa_t f|^2s^\al\,\ud s\\
&\le \frac{C}{1-\al}\int_{0}^R|\nabla f(x,s)|^2s^\al\,\ud s.
 \end{split}
\]
Multiplying the above by $t^\lda$ and integrating over $Q_R$, we obtain
\[
\begin{split}
 \int_{Q_R} t^\lda f^2 &\leq C\int_0^Rt^\lda\,\ud t \int_{B_R} \int_{0}^R|\nabla f(x,s)|^2s^\al\,\ud s\,\ud x\\&
\leq C\int_{Q_R} |\nabla f(x,s)|^2 s^\al,
\end{split}
\]
so (\ref{2.6}) follows.

We are going to prove (\ref{2.4}). Let $p\in (1,\frac{2}{1+\al})$. It follows from (\ref{2.5}) and the H\"older inequality that
\[
\begin{split}
 \int_{\pa' Q_R}|f|^2&=  \int_{\pa' Q_R}(|f|^{\frac{2}{p}})^p
\\&\leq \va \int_{Q_R}|\nabla f^{\frac{2}{p}}|^p+C\va^{-\frac{1}{p-1}}\int_{Q_R}|f|^2\\&
=\va(\frac{2}{p})^p \int_{Q_R}|f|^{2-p}t^{-\frac{p\al}{2}}|\nabla f|^pt^{\frac{p\al}{2}} +C\va^{-\frac{1}{p-1}}\int_{Q_R}|f|t^{-\frac{\al}{2}}|f|t^{\frac{\al}{2}}\\&
\leq \va(\frac{2}{p})^p \left(\int_{Q_R}|f|^2t^{-\frac{p\al}{2-p}}\right)^{\frac{2-p}{2}}\Big(\int_{Q_R}|\nabla f|^2t^{\al}\Big)^{\frac{p}{2}}\\&\quad
+C\va^{-\frac{1}{p-1}}\int_{Q_R}\{ \va^{1+\frac{1}{p-1}} |f|^2t^{-\al} +\va^{-1-\frac{1}{p-1}}|f|^2t^{\al}\}\\&
\leq \va C \int_{Q_R}|\nabla f|^2t^{\al} +\frac{C}{\va^{1+\frac{2}{p-1}}}\int_{Q_R}|f|^2t^{\al},
\end{split}
\]
where we used (\ref{2.6}) for $\lda=-\frac{p\al}{2-p}>-1$ and $\lda=-\al>-1$ in the last inequality.
Therefore, we complete the proof.
\end{proof}

\bigskip

\section{Proof of Theorem \ref{thm1}}
\label{s3}
\bigskip

In this section, we will prove the main results by making use of the Moser's iteration.

For $p\in (0,\infty)$ denote
\[
 \| U\|_{L^{p}(t^{1-2\sigma},Q_R)}:=\left(\int_{Q_R}t^{1-2\sigma} U^{p}\right)^{\frac{1}{p}}.
\]

\begin{prop} \label{lem3.1}
Let $U(X)\in H(t^{1-2\sigma},Q_1)$ be a weak solution of
\be\label{3.1}
\begin{cases}
\mathrm{div}(t^{1-2\sigma}\nabla U(X))\geq 0\quad &\mbox{in } Q_1 \\
-\dlim_{t\rightarrow 0^+}t^{1-2\sigma}\pa_tU(x,t)\leq a(x)U(x,0)+b(x)\quad  &\mbox{on } \pa' Q_1.
\end{cases}
\ee
Then
\[
\dsup_{Q_{1/2}}U^+\leq C(\|U^+\|_{L^{2}(t^{1-2\sigma},Q_R)}+\|b\|_{L^{\infty}(B_1)}),
\]
where $U^+=\max\{0,U\}$, and $C>0$ depends only on $n,\sigma, \|a\|_{L^{\infty}(B_1)}$.
\end{prop}

\begin{proof}
Let $k,m>0$ be some constants. Set $\overline U=U^++k$ and
\[
  \overline U_m=
\begin{cases}
\overline{U}&\quad \mbox{if } U<m,\\
k+m &\quad \mbox{if } U\geq m.
\end{cases}
\]
Consider the test function
\[
 \phi=\eta^2(\overline U_m^\beta\overline U-k^{\beta+1})\in H(t^{1-2\sigma}, Q_1),
\]
for some $\beta\geq 0$ and some nonnegative function $\eta\in C^1_c(Q_1\cup \pa' Q_1)$. Clearly, $\nabla \overline U_m=0$
in $\{U<0\}$ and $\{U\geq m\}$. A direct calculation yields
\[
\begin{split}
 \nabla \phi &=\beta \eta^2 \overline U_m^{\beta-1}\nabla \overline U_m\overline u+\eta^2 \overline U^\beta_m\nabla \overline U
+2\eta \nabla \eta(\overline U^\beta_m\overline U-k^{\beta+1})\\&
=\eta^2 \overline U^\beta_m(\beta \nabla \overline U_m+\nabla \overline U)++2\eta \nabla \eta(\overline U^\beta_m\overline U-k^{\beta+1}).
\end{split}
\]
Multiplying (\ref{3.1}) by $\phi$ and integrating by parts, we have
\be\label{3.2}
\begin{split}
 0&\leq -\int_{Q_1}t^{1-2\sigma}\nabla U\nabla \phi + \int_{\pa'Q_1}a(x)U\phi+b(x)\phi\\
&=-\int_{Q_1}t^{1-2\sigma}\eta^2\overline U^\beta_m(\beta|\nabla \overline U_m|^2+|\nabla \overline U|^2)
-2\int_{Q_1}t^{1-2\sigma}\eta (\overline U^\beta_m \overline U-k^{\beta+1})\nabla \eta\nabla \overline U\\&
\quad +\int_{\pa'Q_1}a(x)U \eta^2(\overline U_m^\beta\overline U-k^{\beta+1})+b(x)\eta^2(\overline U_m^\beta\overline U-k^{\beta+1})\\&
\leq -\frac{1}{2}\int_{Q_1}t^{1-2\sigma}\eta^2\overline U^\beta_m(\beta|\nabla \overline U_m|^2+|\nabla \overline U|^2)+
4\int_{Q_1}t^{1-2\sigma}\overline U^\beta_m \overline U^2|\nabla \eta|^2\\&
\quad +\int_{\pa'Q_1}|a(x)| \eta^2\overline U_m^\beta\overline U^2+|b(x)|\eta^2\overline U_m^\beta\overline U,
\end{split}
\ee
where we used the Cauchy inequality and the fact
$\overline U_m^\beta\overline U-k^{\beta+1}< \overline U_m^\beta\overline U$. Choosing $k=\|b\|_{L^\infty(B_1)}$ if
$b$ is not identically zero. Otherwise choose an arbitrary $k>0$ and eventually let $k\rightarrow 0$.
Then we see that $|b(x)|\eta^2\overline U_m^\beta\overline U\leq \eta^2\overline U_m^\beta\overline U^2$. Hence (\ref{3.2}) gives
\[
\begin{split}
 &\int_{Q_1}t^{1-2\sigma}\eta^2\overline U^\beta_m(\beta|\nabla \overline U_m|^2+|\nabla \overline U|^2)\\&
\leq 8\int_{Q_1}t^{1-2\sigma}\overline U^\beta_m \overline U^2|\nabla \eta|^2+2(\|a\|_{L^\infty(B_1)}+1)\int_{\pa'Q_1} \eta^2\overline U_m^\beta\overline U^2.
\end{split}
\]

Set $W=\overline U^{\frac{\beta}{2}}_m\overline U$. Then
\[
 |\nabla W|^2\leq (1+\beta)(\beta \overline U_m^\beta|\nabla \overline U_m|^2+\overline U_m^\beta|\nabla \overline U|^2).
\]
Therefore, we have
\[
 \int_{Q_1}t^{1-2\sigma}\eta^2|\nabla W|^2\leq C(1+\beta)\left\{ \int_{Q_1}t^{1-2\sigma}W^2|\nabla \eta|^2+ \int_{\pa'Q_1}\eta^2 W^2\right\},
\]
or
\[
 \int_{Q_1}t^{1-2\sigma}|\nabla(\eta W)|^2\leq C(1+\beta)\left\{ \int_{Q_1}t^{1-2\sigma}W^2|\nabla \eta|^2+ \int_{\pa'Q_1}\eta^2 W^2\right\}.
\]

By Lemma \ref{prop2.3},
\[
 C(1+\beta)\int_{\pa'Q_1}\eta^2 W^2\leq \frac{1}{2}\int_{Q_1}t^{1-2\sigma}|\nabla(\eta W)|^2 +C(1+\beta)^{\delta}\int_{Q_1}t^{1-2\sigma}\eta^2 W^2
\] for some $\delta>1$ depending on $n,\sigma$. It follows that
\[
 \int_{Q_1}t^{1-2\sigma}|\nabla(\eta W)|^2\leq C(1+\beta)^\delta \int_{Q_1}t^{1-2\sigma}(\eta^2+|\nabla \eta|^2)W^2.
\]
By the Sobolev inequality, see Lemma \ref{prop2.2}, we obtain
\[
 \left(\int_{Q_1}t^{1-2\sigma}|\eta W|^{2\chi}\right)^{\frac{1}{\chi}}\leq C(1+\beta)^\delta \int_{Q_1}t^{1-2\sigma}(\eta^2+|\nabla \eta|^2)W^2,
\]
where $\chi=\frac{n+1}{n}>1$. For any $0<r<R\leq 1$, consider an $\eta\in C_c(Q_1\cup \pa'Q_1)$ with $\eta=1$ in $Q_r$ and $|\nabla \eta|\leq 2/(R-r)$.
Thus we have
\[
 \left(\int_{Q_r}t^{1-2\sigma} W^{2\chi}\right)^{\frac{1}{\chi}}\leq C\frac{(1+\beta)^\delta}{(R-r)^2} \int_{Q_R}t^{1-2\sigma}W^2.
\]
or, by the definition of $W$,
\[
 \left(\int_{Q_r}t^{1-2\sigma}\overline U^{\beta\chi}_m \overline U^{2\chi}\right)^{\frac{1}{\chi}}
\leq C\frac{(1+\beta)^\delta}{(R-r)^2} \int_{Q_R}t^{1-2\sigma}\overline U^{\beta}_m \overline U^{2}.
\]
Noting that $\overline{U}_m\leq \overline{U}$, we get
\[
 \left(\int_{Q_r}t^{1-2\sigma}\overline U^{\gamma\chi}_m\right)^{\frac{1}{\chi}}
\leq C\frac{(1+\beta)^\delta}{(R-r)^2} \int_{Q_R}t^{1-2\sigma} \overline U^{\gamma}
\]
provided the integral in the right hand side is bounded. By letting $m\rightarrow \infty$, we conclude that
\[
 \|\overline U\|_{L^{\gamma\chi}(t^{1-2\sigma},Q_r)}\leq \left(C\frac{(1+\beta)^\delta}{(R-r)^2}\right)^{\frac{1}{\gamma}}
\|\overline U\|_{L^{\gamma}(t^{1-2\sigma},Q_R)},
\]
where $C>0$ is a constant depending only $n,\sigma, \|a\|_{L^{\infty}(B_1)}$. As in standard Moser iterating procedure, we then arrive at
\[
 \dsup_{Q_{1/2}}\overline U\leq C\|\overline U\|_{L^2(t^{1-2\sigma}, Q_1)}
\]
or
\[
 \dsup_{Q_{1/2}} U^+\leq C(\| U^+\|_{L^2(t^{1-2\sigma}, Q_1)}+k).
\]
Recalling the definition of $k$, we complete the proof.
\end{proof}

The next lemma is so called weak Harnack inequality.

\begin{prop} \label{lem3.2}
Let $U(X)\in H(t^{1-2\sigma},Q_1)$ be a nonnegative weak solution of
\be\label{3.4}
\begin{cases}
\mathrm{div}(t^{1-2\sigma}\nabla U(X))\leq 0\quad &\mbox{in } Q_1 \\
-\dlim_{t\rightarrow 0^+}t^{1-2\sigma}\pa_tU(x,t)\geq a(x)U(x,0)+b(x)\quad  &\mbox{on } \pa' Q_1.
\end{cases}
\ee
Then for some $p>0$ and any $0<\theta<\tau<1$ we have
\[
\dinf_{\overline Q_{\theta}}U+\|b\|_{L^\infty(Q_1)} \geq C\|U\|_{L^{p}(t^{1-2\sigma},Q_\tau)},
\]
where $C>0$ depends only on $n,\sigma,\theta,\tau, \|a\|_{L^{\infty}(Q_1)}$.
\end{prop}

\begin{proof}
Set $\overline U= U+k>0$, for some positive $k$ to be determined and $V=\overline U^{-1}$. Let $\Phi$ be any nonnegative function in $ H(t^{1-2\sigma},Q_1)$
with compact support in $Q_1\cup \pa'Q_1$. Multiplying both sides of first inequality of (\ref{3.4})  by $\overline U^{-2}\Phi$ and integrating by
parts, we obtain
\[
0\geq -\int_{Q_1}t^{1-2\sigma} \frac{\nabla U\nabla \Phi}{\overline U^{2}}
+2\int_{Q_1}t^{1-2\sigma}\nabla U\nabla \overline U \frac{\Phi}{\overline U^3}+\int_{\pa'Q_1}(aU+b)\overline U^{-2}\Phi.
\]
Note that $\nabla U=\nabla \overline U$ and $\nabla V=-\overline U^2\nabla \overline U$. Therefore, we have
\[
\int_{Q_1}t^{1-2\sigma}\nabla V\nabla \Phi+\int_{\pa'Q_1}\tilde a V\Phi\leq 0,
\]
where
\[
\tilde a=\frac{aU+b}{\overline U}.
\]
Choose $k=\|b\|_{L^\infty(Q_1)}$ if $b$ is not identical zero. Otherwise, choose an arbitrary $k>0$ and eventually let it tend to zero.
Note that $\|\tilde a\|_{L^\infty(Q_1)}\leq \|a\|_{L^\infty(Q_1)}+1$. Therefore Proposition \ref{lem3.1} implies that for any $\tau\in (\theta,1)$
and any $p>0$
\[
\dsup_{Q_\theta} V\leq C\| V\|_{L^p(t^{1-2\sigma},Q_\tau)},
\]
or,
\[
\begin{split}
\dinf_{Q_\theta}\overline U &\geq C\left(\int_{Q_\tau}t^{1-2\sigma}\overline U^{-p} \right)^{-\frac{1}{p}}\\&
=C\left(\int_{Q_\tau}t^{1-2\sigma}\overline U^{-p}\int_{Q_\tau}t^{1-2\sigma}\overline U^{p} \right)^{-\frac{1}{p}}
\left(\int_{Q_\tau}t^{1-2\sigma}\overline U^{p} \right)^{\frac{1}{p}},
\end{split}
\]
where $C>0$  depends only on $n,\sigma, p,\theta,\tau$.

The next key point is to show that there exists some $p_0>0$ such that
\[
\int_{Q_\tau}t^{1-2\sigma}\overline U^{-p_0}\int_{Q_\tau}t^{1-2\sigma}\overline U^{p_0}\leq C,
\]
where $C>0$  depends only on $n,\sigma,\tau$. We are going to show that for any
$\tau<1$
\be\label{3.5}
\int_{Q_\tau} e^{p_0|W|}\leq C,
\ee
where $W=\log\overline U-(\log\overline U)_{0,\tau}$. The idea is as usual. (\ref{3.5}) will follows from John-Nirenberg type
lemma (see \cite{MW}) if $W\in BMO(t^{1-2\sigma}\ud X)$.

We first derive an equation for $W$. Multiplying both sides of first inequality of (\ref{3.4})  by $\overline U^{-1}\Phi$ and integrating by
parts, we obtain
\[
\int_{Q_1}t^{1-2\sigma}|\nabla W|^2 \Phi\leq \int_{Q_1}t^{1-2\sigma}\nabla W\nabla \Phi+\int_{\pa'Q_1}\tilde a \Phi,
\]
where
\[
\tilde a=\frac{aU+b}{\overline U}.
\]
Replace $\Phi$ by $\Phi^2$. It follows from the Cauchy inequality and the Sobolev inequality that
\be\label{3.6}
\int_{Q_1}t^{1-2\sigma}|\nabla W|^2 \Phi^2\leq C \int_{Q_1}t^{1-2\sigma}|\nabla \Phi|^2,
\ee
where $C>0$ depends only on $n,\sigma$. Then for any $Q_{2r}(Y)\subset Q_1$, $Y\in \pa \mathbb{R}^{n+1}_+$, choose $\Phi$ with
\[
\mbox{supp} (\Phi)\subset Q_2r(Y)\cup \pa'Q_2r(Y), \quad \Phi=1\mbox{ in } Q_r(Y)\cup \pa'Q_r(Y),\quad |\nabla \Phi|\leq \frac{C}{r}.
\]
We have
\[
 \int_{Q_r(Y)}t^{1-2\sigma}|\nabla W|^2\leq \frac{C}{r^2}\int_{Q_{r}(Y)}t^{1-2\sigma}.
\]
Hence the Poincar\'e inequality, Lemma \ref{prop2.2}, implies
\[
\begin{split}
 \left(\int_{Q_{r}(Y)}t^{1-2\sigma}\right)^{-1}&\int_{Q_{r}(Y)}t^{1-2\sigma}|W-W_{Y,r}|\\&
\leq \left(\int_{Q_{r}(Y)}t^{1-2\sigma}\right)^{-1/2}
\left(\int_{Q_{r}(Y)}t^{1-2\sigma}|W-W_{Y,r}|^2\right)^{1/2}\\&
\leq r\left(\int_{Q_{r}(Y)}t^{1-2\sigma}\right)^{-1/2}
\left(\int_{Q_{r}(Y)}t^{1-2\sigma}|\nabla W|^2\right)^{1/2}\\&
\leq C.
\end{split}
\]
For other $Y\in Q_1$, one can verify the above similarly. Therefore, we conclude that $W\in BMO(t^{1-2\sigma}, Q_1)$.
\end{proof}

\begin{proof}[Proof of Theorem {\rm\ref{thm3.1}}] The proof follows from Proposition 3.1 and 3.2.
\end{proof}

\begin{proof}[Proof of Theorem {\rm\ref{thm1}}]
Since $u\geq 0$ in $\mathbb{R}^n$ be a solution of (\ref{1.1}), there exists a
nonnegative function $U(x,t)\in H(t^{1-2\sigma}, \mathbb{R}^{n+1}_{+})$
satisfying
\[
\mbox{div}(t^{1-2\sigma}\nabla U(x,t))=0\; \text{in}\;\mathbb{R}^{n+1}_{+}
\]
 and $U(x,0)=u(x).$
It follows from (\ref{eqn-DNm})  that
\[
\begin{split}
\lim_{t\rightarrow 0^+}t^{1-2\sigma}\pa_tU(x,t)&=-c_\sigma(-\Delta)^{\sigma}U(x,0)\\
&=-c_\sigma(a(x)U(x,0)+b(x)),
\end{split}
\]
 where we used $u\in C^2(B_1)$. Hence Theorem \ref{thm1} immediately follows from Theorem~\ref{thm3.1}.
\end{proof}

 \begin{thm} Let $0<\sigma<1$ and $B_{R}=B_{R}(0)\subset \mathbb{R}^{n}$, $n>2\sigma$.
Suppose that $a(x)\in L^{\infty}(\mathbb{R}^{n})$,
$0\le u\in C(\mathbb{R}^{n})$ satisfies
\[
(-\Delta )^{\sigma}u(x)=a(x)u(x),\quad x\in B_{R}.
\]
Then for $\delta>0$, there exists $C(n,\sigma,\delta)>0$ such that
\[
\max_{\overline{B}_{R-\delta}}u\le C(n,\sigma,\delta)\min_{\overline{B}_{R-\delta}}u.
\]
\end{thm}
\begin{proof} By rescaling, we can prove it from Theorem \ref{thm1}. See another proof in
\cite{CS10}.
 \end{proof}

\bigskip

\bigskip

 \noindent {\bf Acknowledgements:} 
 Both authors would like to thank Prof. J. Bao and Prof. Y.Y. Li for their encouragement.

\end{document}